\documentclass[11pt]{amsart}

\usepackage{amsmath,amssymb,amsfonts,graphics}
\usepackage[all]{xy}

\newtheorem{thm}{Theorem}[section]
\newtheorem{cor}[thm]{Corollary}
\newtheorem{lem}[thm]{Lemma}
\newtheorem{prop}[thm]{Proposition}
\theoremstyle{definition}
\newtheorem{defn}[thm]{Definition}
\theoremstyle{remark}
\newtheorem{rem}[thm]{Remark}
\newtheorem{exm}[thm]{Example}
\numberwithin{equation}{section}

\newcommand{\R}{\mathbb{R}}
\newcommand{\N}{\mathbb{N}}

\newcommand{\T}{\mathbb{T}}
\newcommand{\C}{\mathbb{C}}

\newcommand{\fA}{\mathcal{A}}
\newcommand{\fB}{\mathcal{B}}
\newcommand{\fC}{\mathcal{C}}
\newcommand{\fR}{\mathcal{R}}
\newcommand{\fS}{\mathcal{S}}

\newcommand{\Fa}{\mathcal{F}_\fA}

\newcommand{\Hc}{\mathcal{H}}

\newcommand{\om}{\omega}

\newcommand{\Oma}{\Omega_\fA}

\newcommand{\supp}{\text{supp}}

\newcommand{\Sp}{\text{Sp}}
\newcommand{\pf}{\text{PF}}

\newcommand{\la}{\langle}
\newcommand{\ra}{\rangle}

\begin{document}

\title[]
{Quasi-Hermitian locally compact groups are amenable}

\author{Ebrahim Samei}
\address{Department of Mathematics and Statistics, University of Saskatchewan, Saskatoon, Saskatchewan, S7N 5E6, Canada}
\email{samei@math.usask.ca}


\author{Matthew Wiersma}
\address{Department of Mathematical and Statistical Sciences, University of Alberta, Edmonton, Canada}
\email{mwiersma@ualberta.ca}

\footnote{{\it Date}: \today.



The first named author was partially supported by NSERC Grant no. 409364-2015.

The second named author was partially support by NSERC Grant no. 2018-05681.}







\maketitle


\begin{abstract}
	A locally compact group $G$ is \emph{Hermitian} if the spectrum ${\rm Sp}_{L^1(G)}(f)\subseteq\R$ for every $f\in L^1(G)$ satisfying $f=f^*$, and \emph{quasi-Hermitian} if ${\rm Sp}_{L^1(G)}(f)\subseteq\R$ for every $f\in C_c(G)$ satisfying $f=f^*$. We show that every quasi-Hermitian locally compact group is amenable. This, in particular, confirms the long-standing conjecture that every Hermitian locally compact group is amenable, a problem that has remained open since the 1960s. Our approach involves introducing the theory of ``spectral interpolation of triple Banach $*$-algebras'' and applying it to a family ${\rm PF}_p^*(G)$ ($1\leq p\leq \infty$) of Banach $*$-algebras related to convolution operators that lie between $L^1(G)$ and $C^*_r(G)$, the reduced group C$^*$-algebra of $G$. We show that if $G$ is quasi-Hermitian, then ${\rm PF}_p^*(G)$ and $C^*_r(G)$ have the same spectral radius on Hermitian elements in $C_c(G)$ for $p\in (1,\infty)$, and then deduce that $G$ must be amenable. We also give an alternative proof to Jenkin's result in \cite{Jenk1} that a discrete group containing a free sub-semigroup on two generators is not quasi-Hermitian. This, in particular, provides a dichotomy on discrete elementary amenable groups: either they are non quasi-Hermitian or they have subexponential growth. Finally, for a non-amenable  group $G$ with either rapid decay or Kunze-Stein property, we prove the stronger statement that ${\rm PF}_p^*(G)$ is not ``quasi-Hermitian relative to $C_c(G)$'' unless $p=2$.
\end{abstract}

\section{Introduction}

Let $\fA$ be a Banach $*$-algebra, i.e. a (not necessarily unital) Banach algebra equipped with an isometric involution $* : \fA\to \fA$. Recall that $a\in \fA$ is called Hermitian if $a=a^*$. If $\fS$ is any subset of $\fA$, we let
$$ \fS_h=\{a\in \fS : a=a^*\}$$
denote the set of Hermitian elements in $\fS$.
The Banach $*$-algebra $\fA$ is \emph{Hermitian} if
$$ {\rm Sp}_{\fA}(a)\subseteq \R$$
for every $a\in \fA_h$, where $\rm{Sp}_{\fA}(a)$ denotes the spectrum of $a\in \fA$. Hermitian Banach $*$-algebras are one of the most important subclasses of Banach $*$-algebras. For instance, Hermitian Banach $*$-subalgebras, or more generally Hermitian Fr\'{e}chet $*$-subalgebras, of C$^*$-algebras appear in many areas of mathematics such as approximation theory, time-frequency analysis and signal processing, non-commutative geometry and geometric group theory (see \cite{GK1}, \cite{GK2}, \cite{GL}, \cite{black-cuntz}, \cite{rennie}, \cite{Jol}, \cite{Chat} and references therein). We also refer the reader to \cite[Section 11.4]{Pal2} for a very nice survey of the topic (see also \cite[Sections 35 and 41]{BD}).


Let $G$ be a locally compact group. The group algebra $L^1(G)$ is a Banach $*$-algebra with involution given by
$$ f^*(s)=\overline{f(s^{-1})}\Delta(s^{-1})$$
for $f\in L^1(G)$ and $s\in G$, where $\Delta: G\to (0,\infty)$ denotes the modular function of $G$. The locally compact group $G$ is \emph{Hermitian} if $L^1(G)$ is Hermitian. Na\u\i mark initiated the study of Hermitian groups in 1956 by showing that the group algebra of $SL(2,\C)$ is not Hermitian (see \cite{Naimark}). The class of Hermitian locally compact groups includes abelian groups, compactly generated groups of polynomial growth, and the $ax+b$-group (see, for example, \cite{Los}, \cite{FGL}). The literature on Hermitian groups is vast, particularly from the 1960s and 1970s, and we make little attempt at a summary. We instead refer the reader to \cite[p. 1441--1444 and \S 12.6.22]{Pal2} and \cite{Palma1} for nice historical surveys. 

The biggest open problem regarding Hermitian locally compact groups is, arguably, whether every such group is amenable. We provide an affirmative answer to this long-standing open problem in this paper. The authors are unaware of when this problem first arose, but it was certainly being considered by 1963 when Hulanicki announced the result was true for discrete groups in \cite{Hul Announce} and that a proof would appear in a forthcoming paper. Unfortunately the proof was flawed and the forthcoming paper referred to in \cite{Hul Announce} never appeared in the literature.

The question of whether every Hermitian locally compact group is amenable has previously best been answered for the class of almost connected groups in 1978.
\begin{thm}[Palmer \cite{Pal Rocky}]
	A Hermitian almost connected locally compact group is amenable.
\end{thm}
\noindent This result was deduced from Jenkin's characterization of Hermitian reductive Lie groups (see \cite{Jenk2}). The best known result on the topic for discrete groups was the following.
\begin{thm}[Jenkins \cite{Jenk1}]
	A discrete group containing a free sub-semigroup on two generators is not Hermitian.
\end{thm}
\noindent In particular, every discrete group containing a non-commutative free group is not Hermitian. Beyond the realm of non-amenable groups, this theorem also implies the discrete group $\mathbb Q\rtimes\mathbb Q^*$ is not Hermitian.

More recently, in \cite{Palma1}, Palma made a systematic study of Hermitian locally compact group using the concept of the ``capacity" of a Hermitian element of a Banach $*$-algebra. His work was motivated by the work
of Jenkins (see \cite{Jenk2}) and Fountain, Ramsay and Williamson (see \cite{FRW}). He provided many examples of non-Hermitian locally compact groups including certain totally disconnected groups and certain torsion groups.

The main result of this paper applies to a class of locally compact groups which we call quasi-Hermitian.
\begin{defn}\label{D:qH-group}
	A locally compact group $G$ is \emph{quasi-Hermitian} if
	$\Sp_{L^1(G)}(f)\subseteq \R$
	for every $f\in C_c(G)_h$.
\end{defn}
\noindent The class of quasi-Hermitian locally compact groups is strictly larger than Hermitian locally compact groups, and will be discussed this in more detail later in the paper. The most significant part of our main result (Theorem \ref{T:quasi-Hermitian implies amenable}) is the following theorem.

\begin{thm}\label{Thm:IntroMain}
A locally compact group $G$ is quasi-Hermitian if and only if for every $f\in C_c(G)$, $\Sp_{L^1(G)}(f)=\Sp_{C^*_r(G)}(f)$. In particular, a quasi-Hermitian locally compact group is amenable.
\end{thm}
\noindent This, in particular, implies that every Hermitian locally compact group is necessarily amenable. Our proof of Theorem \ref{Thm:IntroMain} is surprisingly short and involves a completely new approach to the problem, which we summarize below.


We begin with an analysis and discussion of spectral properties for Banach $*$-algebras. We then consider a family of triple Banach $*$-algebras $\fA \subseteq \fB \subseteq \fC$ whose spectral radii satisfy an interpolation relation; we call it ``spectral interpolation of triple Banach $*$-algebras" (Definition \ref{D:interpolation triple Banach $*$-algebras}). For such a family of algebras, we present conditions that ensure $r_\fB(a)=r_\fC(a)$ for ``many'' $a\in\fA_h$  (see Theorem \ref{T:interpolation triple-general alg}). Subsequently, for a locally compact group $G$, we consider a class of Banach $*$-algebras ${\rm PF}^*_p(G)$ ($1\leq p\leq \infty$) of convolution operators, introduced in \cite{KY} and \cite{LY}, and show that for $1<p<2$, $(L^1(G),\pf_p^*(G),C^*_r(G))$ is a spectral interpolation of triple $*$-semisimple Banach $*$-algebras (Definition \ref{D:involutive pseudofunctions} and Corollary \ref{C:Interpolation triple-pseudofunctions-L1 case}). We then apply our method to show if $G$ is quasi-Hermitian, then for every $1<p<2$ and $f\in C_c(G)_h$,
$$r_{{\rm PF}_p^*(G)}(f)=r_{C^*_r(G)}(f).$$
Thus, by letting $p\to 1^+$ and making a careful analysis of the spectral radii relations between these algebras, we deduce that $r_{L^1(G)}(f)=r_{C^*_r(G)}(f)$ for all $f\in C_c(G)_h$ (Theorem \ref{T:quasi-Hermitian implies amenable}). This, in particular, implies that the full and reduced C$^*$ algebra of $G$ must coincides so that $G$ must be amenable. Moreover, this also implies that a discrete group containing a free sub-semigroup on two generators is not quasi-Hermitian thus providing an alternative proof to the result of Jenkin in \cite{Jenk1} (see Remark \ref{R:free sub semigroup-not qH}). In particular, using the structure theory of discrete elementary amenable groups, we obtain that a discrete elementary amenable group is quasi-Hermitian if and only if it has a subexponential growth (Corollary \ref{C:qH-elementary amenable}).

The paper ends with asking whether an analogue of Theorem \ref{Thm:IntroMain} holds for $\pf_p^*(G)$ when $p\neq 1,2,\infty$. Solutions are obtained when $G$ either has the rapid decay property or is a Kunze-Stein group.

\section{Spectral properties of Banach $*$-algebras}


\subsection{Definitions and basic results}\label{subsec:basic}
We begin by recalling some background on Banach $*$-algebras. The reader should see \cite{Pal1} and \cite{Pal2} for a more comprehensive treatment of the topic.

If $X$ is any Banach space, $B(X)$ will denote the space of bounded linear operators on $X$. Let $\fA$ be a Banach $*$-algebra. The {\it reducing ideal} of $\fA$ is
$$\fA_\fR:=\bigcap \{\ker \pi \mid \Hc\text{ is a Hilbert space and }\pi:\fA\to B(\Hc)\ \text{is a $*$-representation}\}.$$
Note that no assumptions of continuity are imposed on $\pi$ in this definition since $*$-representations of Banach $*$-algebras are automatically contractive.
The reducing ideal $\fA_\fR$ is a closed $*$-ideal of $\fA$ that is known to contain the Jacobson radical $\fA_J$ of $\fA$.
The Banach $*$-algebra $\fA$ is {\it $*$-semisimple} if $\fA_\fR=\{0\}$.
Note that $\fA/\fA_\fR$ is always $*$-semisimple since $*$-representations of $\fA$ descend to $*$-representations of $\fA/\fA_\fR$.

Suppose $\fA$ is a $*$-semisimple Banach $*$-algebra. The {\it enveloping C$^*$-algebra} of $\fA$, denoted by $C^*(\fA)$, is the unique C*-algebra $\fB$ which admits the following universal property: there exists an injective $*$-homomorphism $\pi_u:\fA\to \fB$ with dense range so that for every $*$-representation $\pi: \fA\to B(\Hc)$, there exists a $*$-representation $\tilde{\pi}:\fB \to B(\Hc)$ so that $\pi=\tilde{\pi}\circ \pi_u$.



Let $\fA$ be a commutative Banach algebra. We write $\Oma$ to denote the {\it spectrum} of $\fA$, i.e. the set of all nonzero multiplicative linear functionals on $\fA$. The spectrum $\Oma$ is a locally compact Hausdorff topological space when equipped with the $w^*$-topology induced from $\fA^*$. The {\it Gelfand representation} of $\fA$ is given by
\begin{equation*}
  \Fa: \fA \to C_0(\Oma) \ \ ,\ \Fa(a)(\varphi)=\varphi(a) \ \ (a\in \fA, \varphi \in \Oma).
\end{equation*}
The kernel of $\Fa$ is exactly $\fA_J$, the Jacobson radical of $\fA$. In particular, the Banach algebra $\fA$ is semisimple if and only if its Gelfand representation is injective.
Now suppose $\fA$ is also a Banach $*$-algebra.
Since the image ${\rm Im}\,\Fa$ of $\Fa$ is a subalgebra of $C_0(\Oma)$ that separates points of $\Oma$ and vanishes nowhere on $\Oma$,
${\rm Im}\,\Fa$ is dense in $C_0(\Oma)$ whenever $\Fa$ is a $*$-homomorphism by the Stone-Weierstrass theorem. So when is $\Fa$ is $*$-homomorphism? This occurs exactly when $\fA$ is Hermitian which follows from the fact that $\Fa$ is an algebra homomorphism, every $a\in \fA$ can be written in the form $a=a_1+ia_2$ for $a_1,a_2\in \fA_h$, and
\begin{equation}\label{Eq:specrum vs image of Gelfand rep}
  \Sp_\fA(a)\setminus \{0\}=\{\varphi(a): \varphi \in \Oma \}\setminus \{0\}\ \ \ (a\in \fA).
\end{equation}
Finally note Equation \eqref{Eq:specrum vs image of Gelfand rep} immediately implies
\begin{equation}\label{Eq:spec. radius vs infinity norm Gelfand rep}
  r_\fA(a)=\|\Fa(a)\|_{C(\Oma)}=\sup \{|\varphi(a)|: \varphi \in \Oma \}
\end{equation}
for every $a\in\fA$.

\subsection{Invariant spectral radius and spectral subalgebras}

\begin{defn}\label{D:nested pair Ban alg}
  We say that $\fA\subseteq \fB$ is a {\it nested pair} of Banach $*$-algebras if $\fA$ and $\fB$ are Banach $*$-algebras and $\fA$ embeds continuously into $\fB$ as a dense $*$-subalgebra. If, in addition, $\fB$ (and hence $\fA$) is $*$-semisimple, we say $\fA\subseteq \fB$ is a {\it nested pair of $*$-semisimple} Banach $*$-algebras. A {\it nested triple of $(*$-semisimple$)$} Banach $*$-algebras is defined similarly.
\end{defn}

Inspired by \cite[Definition 2.5.1 and Proposition 2.5.2]{Pal1}, we give the following definitions.

\begin{defn}\label{D:Inv spec radius-Spectral subalg}
  Let $\fA\subseteq \fB$ be a nested pair of Banach $*$-algebras and $\fS$ a (not necessarily closed) $*$-subalgebra of $\fA$. We say $\fS_h$ has {\it invariant spectral radius} in $(\fA,\fB)$ if
  \begin{equation*}
    r_\fA(a)=r_\fB(a).
  \end{equation*}
  for every $a\in \fS_h$.
  If $\fA_h$ has invariant spectral radius in $(\fA,\fB)$, we simply say that $\fA_h$ has {\it invariant spectral radius} in $\fB$.
  Similarly, $\fS$ is a {\it spectral subalgebra} of $(\fA,\fB)$ if
  \begin{equation*}
    \Sp_\fA(a)\cup \{0\}=\Sp_\fB(a)\cup \{0\}.
  \end{equation*}
  for every $a\in \fS$.
  We simply say $\fA$ is a {\it spectral subalgebra} of $\fB$ when $\fA$ is a spectral subalgebra of $(\fA,\fB)$.
\end{defn}

  Clearly if $\fS$ is a spectral subalgebra of $(\fA,\fB)$, then $\fS_h$ has an  invariant spectral radius in $(\fA,\fB)$. The following theorem of Barnes provides a partial converse to this statement. Indeed, suppose $\fA\subseteq\fB$ is a nested pair of Banach $*$-algebras where $\fB$ is a C*-algebra, and $\fS$ is a $*$-subalgebra of $\fA$. The Barnes-Hulanicki Theorem implies $\fS$ is a spectral subalgebra of $(\fA,\fB)$ whenever $\fS_h$ has invariant spectral radius in $(\fA,\fB)$.

\begin{thm}[{\bf Barnes-Hulanicki Theorem} \cite{Branes}]\label{T:Barnes Theorem}
Let $\fA$ be a Banach $*$-algebra, $\fS$ a $*$-subalgebra of $\fA$, and 
$\pi : \fA\to B(\Hc)$ a faithful $*$-representation. If $\fA$ is unital, we assume that $\pi(1_\fA)=id_{B(H)}$.
If
$$r_\fA(a)=\|\pi(a)\|$$
for all $a\in \fS_h$,
then
$$\Sp_\fA(a)=\Sp_{B(\Hc)}(\pi(a))$$
for every $a\in \fS$.
\end{thm}

The above theorem of Barnes slightly generalizes a well known and frequently used result of Hulanicki (see \cite[Proposition 3.5]{Hul3}). A further generalization can be found in \cite[Lemma 3.1]{FGL}.

The conditions of invariant spectral radius and spectral subalgebra are satisfied in many examples of nested Banach $*$-algebras. We pause to note one important class of examples.

\begin{exm}For a nested pair $\fA\subseteq \fB$ of Banach $*$-algebras, we say $\fA$ is a {\it differential subalgebra} of $\fB$ if there exists $K>0$ and $0\leq \theta <1$ such that
$$\|a^2\|_\fA\leq K\|a\|^{1+\theta}_\fA\|a\|^{1-\theta}_\fB \ \ \ (a\in \fA).$$
It is well known that if $\fA$ is a differential subalgebra of $\fB$, then $\fA_h$ has invariant spectral radius in $\fB$. If we additionally assume that $\fB$ is Hermitian, then $\fA$ is automatically a spectral subalgebra of $\fB$ so that it is also Hermitian (see \cite[Lemma 3.2]{GK1}). Differential subalgebras naturally occur in many contexts. For example, differential subalgebras of C*-algebras arise in time-frequency analysis and in noncommutative geometry (e.g., see \cite{GK1}, \cite{black-cuntz} and \cite{rennie}). Further, this notion has applications to K-theory since spectral subalgebras induce isomorphism in the K-theory (e.g., see \cite{Jol} and \cite{Laff}).
\end{exm}

The following provides an explicit example of a differential subalgebra that we will return to later in the paper.

\begin{exm}[Pytlik]\label{Ex:Pytlik1}
	Let $G$ be a locally compact group. A (submultiplicative) symmetric weight $\om:G\to [1,\infty)$ is {\it weakly additive} if there is $C>0$ such that
	$$\om(st)\leq C(\om(s)+\om(t))$$
	for all $s,t\in G$. Fix a symmetric weakly additive weight $\omega$ on $G$. Pytlik shows in \cite[Lemma 1 and Lemma 2]{Pytlik} that the Beurling algebra $L^1(G,\om)$ is a differential subalgebra of $L^1(G)$.
\end{exm}

%

We finish this subsection with a result that demonstrates one of many useful properties related to an invariant spectral radius.

\begin{prop}\label{P:dense Spec subalg-same C envelope}
Let $\fA\subseteq \fB$ be a nested pair of $*$-semisimple Banach $*$-algebras, and $\fS$ be a dense $*$-subalgebra of $\fA$. Suppose $\fS_h$ has an invariant spectral radius in $(\fA,\fB)$. Then $\fA$ and $\fB$ have the same $C^*$-envelope. In particular, if $\fB$ is a C$^*$-algebra, then $\fB$ is the C$^*$-envelope of $\fA$.
\end{prop}

\begin{proof}
Let $C^*(\fA)$ and $C^*(\fB)$ be the enveloping $C^*$-algebras of $\fA$ and $\fB$, respectively. Since any bounded $*$-representation of $\fB$ restricts to a $*$-representation of $\fA$ and also $\fA$ is dense in $\fB$, the inclusion $\fA\subseteq \fB$ extends to a surjective $*$-homomorphism $L_1: C^*(\fA) \to C^*(\fB)$.
Let $\pi_u: \fA\to C^*(\fA)$ be the canonical inclusion of $\fA$ inside $C^*(\fA)$. Then, for every $a\in \fS$,
\begin{eqnarray*}
\|\pi_u(a)\|^2 &=& \|\pi_u(a^*a)\| \\
&=& r_{C^*(\fA)}(\pi_u(a^*a)) \\
&\leq &  r_{\fA}(a^*a) \\
&=&  r_{\fB}(a^*a) \\
&\leq & \|a^*a\|_\fB \\
&\leq & \|a\|_\fB^2.
\end{eqnarray*}
Thus $\pi_u$ extends to a $*$-homomorphism from $\fB$ into $C^*(\fA)$. By the universal property of $C^*(\fB)$, the identity map on $\fA$ extends to surjective $*$-homomorphism $L_2:C^*(\fB)\to C^*(\fA)$. Set
$$L:C^*(\fA) \to C^*(\fA) \ \ \ ,\ \ \  L:=L_2\circ L_1.$$
Clearly $L$ is a surjective $*$-homomorphism on $C^*(\fA)$. Moreover, since $L(\pi_u(a))=\pi_u(a)$ for every $a\in \fA$ and $\pi_u(\fA)$ is dense in $C^*(\fA)$, it follows that $L$ is the identity map on $C^*(\fA)$. In particular, $L_1$ is injective so that it is an isometric $*$-isomorphism from $C^*(\fA)$ onto $C^*(\fB)$.
\end{proof}

\subsection{Hermitian and quasi-Hermitian Banach $*$-algebra}

We begin by noting some conditions which are equivalent to a Banach $*$-algebra being Hermitian.


\begin{lem}\label{T:Hermition Ban alg-equivalent conditions}
  The following are equivalent for a Banach $*$-algebra $\fA$.\\
  $(i)$ $\fA$ is Hermitian;\\
  $(ii)$ $\fA$ is symmetric, i.e. $\Sp_\fA(a^*a)\subseteq [0,\infty)$ for every $a\in \fA$.

  If $\fA$ is $*$-semisimple, then $(i)$ and $(ii)$ are equivalent to each of the following.\\
  $(iii)$ $\fA$ is a spectral $*$-subalgebra of $C^*(\fA)$;\\
  $(iv)$ $r_{\fA}(a)=r_{C^*(\fA)}(a)$ for every $a\in \fA$.

    If $\fA$ is commutative, then conditions $(i)$ and $(ii)$ are equivalent to the following.\\
   $(v)$ the Gelfand representation $\Fa$ is a $*$-homomorphism. 

  If $\fA$ is commutative and $*$-semisimple, then conditions $(i)$ through $(v)$ are equivalent to the following.\\
	$(vi)$ Let $\pi_u$ be the universal embedding of $\fA$ into $C^*(\fA)$. The Gelfand transform $\Fa : \fA\to C_0(\Oma)$ extends to an isometric $*$-isomorphism $\tilde{\Fa}:C^*(\fA)\to C_0(\Oma)$ such that $\Fa=\tilde{\Fa}\circ \pi_u$.
\end{lem}

\begin{proof}
  The equivalence of $(i)$ and $(ii)$ is the well known Shirali-Ford Theorem (see \cite[Theorem 11.4.1]{Pal2}).

  Suppose $\fA$ is $*$-semisimple. The equivalence of $(i)$ with $(iii)$ is proven in \cite[Theorem 11.4.1]{Pal2}.
  The implication $(iii)$ $\Longrightarrow$ $(iv)$ is clear. The reverse implication follows from the Barnes-Hulanicki Theorem (see Theorem \ref{T:Barnes Theorem}).

 Now suppose $\fA$ is commutative. The equivalence of $(i)$ and $(v)$ is well known and has already been observed in Subsection \ref{subsec:basic}. Suppose $\fA$ is also $*$-semisimple. The equivalence of $(v)$ and $(vi)$ is easily deduced from Equation \eqref{Eq:specrum vs image of Gelfand rep} and Proposition \ref{P:dense Spec subalg-same C envelope}.
\end{proof}

We now introduce the notion of a quasi-Hermitian $*$-subalgebra of a Banach $*$-algebra, which generalizes that of a Hermitian Banach $*$-algebra.

\begin{defn}\label{D:quasi-Hermitian Ban alg}A dense $*$-subalgebra $\fS$ of a Banach $*$-algebra $\fA$ is {\it quasi-Hermitian} in $\fA$ if $\Sp_\fA(a)\subseteq \R$ for every $a\in \fS_h$.
\end{defn}

\noindent Definition \ref{D:qH-group} can be rephrased as stating that a locally compact group $G$ is quasi-Hermitian if and only if $C_c(G)$ is quasi-Hermitian in $L^1(G)$.

Let $\fS$ be a dense $*$-subalgebra of a Banach $*$-algebra. The property of $\fS$ being quasi-Hermitian is ``relative'' in the sense that it depends upon which Banach $*$-algebra $\fS$ is viewed as being contained in. For example, let $\fA$ be any non-Hermitian $*$-semisimple Banach $*$-algebra. Then $\fA$ is a quasi-Hermitian $*$-subalgebra of $C^*(\fA)$, but $\fA$ is not a quasi-Hermitian $*$-subalgebra of itself.

Suppose $\fS$ is a dense $*$-subalgebra of a commutative Banach $*$-algebra $\fA$. We show in the following proposition that $\fA$ is Hermitian if and only if $\fS$ is quasi-Hermitian. The proof is straightforward but the consequences are essential.

\begin{prop}\label{T:Hermitian vs quasi Hermitian-comm alg}
  The following are equivalent for a commutative Banach $*$-algebra $\fA$ with dense $*$-subalgebra $\fS$.\\
  $(i)$ $\fA$ is Hermitian;\\
  $(ii)$ $\fS$ is quasi-Hermitian in $\fA$;\\
  $(iii)$ for every $a\in \fS$, $\Sp_\fA(a^*a)\subseteq [0,\infty)$.\\
   \end{prop}

   \begin{proof}
   The implication $(i)\Longrightarrow (iii)$ is provided by Lemma \ref{T:Hermition Ban alg-equivalent conditions}.

   The proof of $(iii)\Longrightarrow (ii)$ is identical to the ``easy part" of the Shirali-Ford Theorem. For the sake of self-containment, we include it here. Let $a\in\fS_h$. If $(iii)$ holds, then
   $$ \{\alpha^2: \alpha \in \Sp_{\fA}(a)\}\subseteq \Sp_{\fA}(a^2)=\Sp(a^*a)\subseteq [0,\infty).$$
   Hence $\Sp_{\fA}(a)\subseteq \R$ for every $a\in \fS_h$.

	We argued in Subsection \ref{subsec:basic} that the Gelfand transform is a $*$-homomorphism if and only if $\fA$ is Hermitian. The same argument shows the restriction of $\Fa$ to $\fS$ is a $*$-homomorphism if and only if $\fS$ is quasi-Hermitian in $\fA$. Thus, the implication $(ii)\Longrightarrow(iii)$ follows by density of $\fS$ in $\fA$ and continuity of $\Fa$.
   \end{proof}

We point out that the results of Proposition \ref{T:Hermitian vs quasi Hermitian-comm alg} may not hold when the assumption of commutativity is dropped. Namely, a non-Hermitian Banach $*$-algebra may contain a quasi-Hermitian dense $*$-subalgebra. As demonstrated by the following example due to Pytlik, such dense $*$-subalgebras may even be a Hermitian Banach $*$-algebras under a different norm.

\begin{exm}\label{E:Hermitian- weighted locally finite alg}
Let $G$ be a locally finite, countable, discrete group, i.e. a discrete group containing an increasing sequence  $\{G_i\}_{i\in \N}$ of finite subgroups such that $G=\bigcup_{i\in \N} G_i$.
Take an increasing sequence $\{n_i\}_{i\in \N}$ of natural numbers and define
$\om:G \to [1,\infty)$ by
$$\om=1+\sum_{i=1} n_i 1_{G_{i+1}\setminus G_{i}}.$$
It is easy to see that
$$\om(st)=\max\{ \om(s),\om(t) \}$$
for every $s,t\in G$.
In particular, $\om$ is a weakly additive symmetric weight on $G$. Choose $\{n_i\}$ so that $1/\om \in \ell^1(G)$. Then, as we will elaborate on in Remark \ref{R:Spec interpolation triple-Pytlik}, $\ell^1(G,\om)$ is Hermitian. In particular, $\ell^1(G,\omega)$ is quasi-Hermitian when viewed as a dense $*$-subalgebra of $\ell^1(G)$. Since there are examples of locally finite, countable, discrete groups that are not Hermitian (see \cite{Hul4}), we deduce that a dense quasi-Hermitian $*$-subalgebra of a non-Hermitian Banach $*$-algebra may even be a Hermitian Banach $*$-algebra with respect to a different norm.
\end{exm}


\section{Spectral interpolation of triple Banach $*$-algebras}

We now define an interpolation condition for the spectral radius of Banach $*$-algebras. The theory developed around this condition is the main ingredient in our investigation of Banach $*$-algebras containing quasi-Hermitian dense $*$-subalgebras.


\begin{defn}\label{D:interpolation triple Banach $*$-algebras}
	Suppose $\fA\subseteq \fB \subseteq \fC$ is a nested triple of Banach $*$-algebras and $\fS$ is a dense $*$-subalgebra of $\fA$. We say $(\fA,\fB,\fC)$ is a {\it spectral interpolation of triple Banach $*$-algebras relative to $\fS$} if
  \begin{equation}\label{Eq:interpolation triple spec raduis relation}
    r_\fB(a) \leq r_\fA(a)^{1-\theta}r_\fC(a)^\theta
  \end{equation}
  for every $a\in \fS_h$.
\end{defn}


The following remark links Example \ref{Ex:Pytlik1} with spectral interpolation of triple Banach $*$-algebras.

\begin{rem}\label{R:Spec interpolation triple-Pytlik}
The condition \eqref{Eq:interpolation triple spec raduis relation} in Definition \ref{D:interpolation triple Banach $*$-algebras} appears in \cite{Pytlik} where Pytlik shows in the proof of \cite[Theorem 1]{Pytlik} that $(L^1(G,\om),L^1(G), C^*_r(G))$ is a spectral interpolation of triple Banach $*$-algebras relative to $L^1(G,\om)$ whenever $G$ is a locally compact group and $\omega$ is a symmetric weakly additive weight on $G$ such that $\frac{1}{\om}\in L^p(G)$ for some $0<p<\infty$. From this, he deduces
\begin{equation*}
r_{L^1(G,\om)}(f)=r_{L^1(G)}(f)=r_{C^*_r(G)}(f),
\end{equation*}
for every $f\in L^1(G,\om)$ and, hence, that $L^1(G,\omega)$ is Hermitian.
%
\end{rem}

We use spectral interpolation of triple Banach $*$-algebras to study a class of Banach $*$-algebras related to convolution operators in the next section.

The main application of relation \eqref{Eq:interpolation triple spec raduis relation} for a nested triple $\fA\subseteq\fB\subseteq\fC$ in the existing literature is to show that the spectral radii of the three algebras coincide on $\fA$ whenever the spectral radii of $\fA$ and $\fB$ coincide on $\fA$. This is how Pytlik applied condition \eqref{Eq:interpolation triple spec raduis relation}  in \cite{Pytlik}, as elaborated on in the previous remark. We will be applying this condition differently since we do not wish to assume equality between $r_{\fA}$ and $r_{\fB}$ on $\fA$, or even on a dense $*$-suablgebra of $\fA$. Instead we are interested in the case when $\fA$ admits a quasi-Hermitian dense $*$-subalgebra $\fS$. Though this initially seems too weak of a condition to garnish much information from, we show it has important consequences when $\fA$, $\fB$, and $\fC$ are commutative. From this, we deduce the equality of $r_\fB$ and $r_\fC$ on $\fS_h$ in the noncommutative case.


\begin{prop}\label{P:interpolation triple-commutative alg}
  The following conditions hold for a spectral interpolation of triple $*$-semisimple commutative Banach $*$-algebras  $(\fA,\fB,\fC)$  relative to a quasi-Hermitian dense $*$-subalgebra $\fS$ of $\fA$.\\
   $(i)$ $\fA$, $\fB$, and $\fC$ are all Hermitian;\\
   $(ii)$ $C^*(\fB)=C^*(\fC)$ and $\mathcal F_\fB$ is the restriction of $\mathcal F_\fC$ to $\fB$;\\
   $(iii)$ $\fB$ is a spectral subalgebra of $\fC$.
\end{prop}

\begin{proof}
  $(i)$ If $a\in \fS_h$, then
  $$\Sp_\fC(a)\cup \{0\}\subseteq \Sp_\fB(a)\cup \{0\} \subseteq \Sp_\fA(a)\cup \{0\} \subseteq \R.$$
  So $\fS$ is also quasi-Hermitian in $\fB$ and $\fC$ and, hence, $\fA$, $\fB$ and $\fC$ are Hermitian by Proposition \ref{T:Hermitian vs quasi Hermitian-comm alg}.

  $(ii)$ Consider the following commutative diagram.
  \begin{displaymath}
    \xymatrix{
        \fA \ar[r]^{id_{\fA,\fB}} \ar[d]_{\pi_\fA} &   \fB \ar[r]^{id_{\fB,\fC}} \ar[d]_{\pi_\fB}  & \fC \ar[d]^{{\pi_\fC}}\\
        C^*(\fA) \ar[r]^{\varphi_{\fA,\fB}}       &  C^*(\fB)  \ar[r]^{\varphi_{\fB,\fC}} & C^*(\fC) }
\end{displaymath}
The maps $id_{\fA,\fB}$ and $id_{\fB,\fC}$ denote the inclusion maps, $\pi_\fA$, $\pi_\fB$, and $\pi_\fC$ denote the canonical inclusions of the given Banach $*$-algebra into its corresponding C$^*$-envelope, and $\varphi_{\fA,\fB}$ and $\varphi_{\fB,\fC}$ are the continuous extensions of $id_{\fA,\fB}$ and $id_{\fB,\fC}$ to $*$-homomorphisms between the corresponding C$^*$-envelopes. We further note both $\varphi_{\fA,\fB}$ and $\varphi_{\fB,\fC}$ are surjective since  these maps have dense images and $*$-homomorphisms between C*-algebras have closed ranges (see \cite[Theorem 10.1.11]{Pal2}). On the other hand, by part $(i)$, $\fA$, $\fB$ and $\fC$ are Hermitian. So, by Lemma \ref{T:Hermition Ban alg-equivalent conditions}$(iii)$,  $\fA$, $\fB$ and $\fC$ are spectral subalgebras of $C^*(\fA)$, $C^*(\fB)$ and $C^*(\fC)$, respectively. Therefore the hypothesis relation \eqref{Eq:interpolation triple spec raduis relation} implies that for every $a\in \fS$,
\begin{eqnarray*}
 \|a\|_{C^*(\fB)}^2 &=&  \|a^*a\|_{C^*(\fB)} \\
   &=&  r_\fB(a^*a) \\
   &\leq & r_\fA(a^*a)^{1-\theta}r_\fC(a^*a)^\theta \\
  &=&  \|a^*a\|_{C^*(\fA)}^{1-\theta} \|a^*a\|_{C^*(\fC)}^\theta \\
  &=&   \|a\|_{C^*(\fA)}^{2(1-\theta)} \|a\|_{C^*(\fC)}^{2\theta}.
\end{eqnarray*}
Hence, by the density of $\fS$,
\begin{equation}\label{Eq:interpolation norm relation-C envelopes}
   \|\varphi_{\fA,\fB}(\xi)\|_{C^*(\fB)}\leq \|\xi\|_{C^*(\fA)}^{1-\theta} \|\varphi_{\fA,\fC}(\xi)\|_{C^*(\fC)}^{\theta}
  \end{equation}
  for every $\xi\in C^*(\fA)$, where $\varphi_{\fA,\fC}:=\varphi_{\fB,\fC}\circ \varphi_{\fA,\fB}$.
  It follows from equation \eqref{Eq:interpolation norm relation-C envelopes}
  \ that
$\ker \varphi_{\fA,\fC}=\ker \varphi_{\fA,\fB}$. Hence, $\varphi_{\fB,\fC}$ is injective since $\varphi_{\fA,\fB}$ is surjective. Thus $\varphi_{\fB,\fC}$ is an isometric $*$-isomorphism so that $C^*(\fB)=C^*(\fC)$. The remainder of part $(ii)$ and the verification part $(iii)$ follow from Lemma \ref{T:Hermition Ban alg-equivalent conditions}$(iii)$ and $(vi)$.
\end{proof}

We now deduce the main result of this section. Recall the straightforward fact that if $E$ is a compact subset of $\C$ such that its boundary $\partial E$ lies in $\R$, then $E$ itself must be a subset of $\R$.

\begin{thm}\label{T:interpolation triple-general alg}
   If $(\fA,\fB,\fC)$ is a spectral interpolation of triple $*$-semisimple Banach $*$-algebras relative to a quasi-Hermitian dense $*$-subalgebra $\fS$ of $\fA$, then $\fS_h$ has invariant spectral radius in $(\fB,\fC)$. In particular, we have the canonical $*$-isomorphism $C^*(\fB)=C^*(\fC)$, and
  \begin{equation}\label{Eq:interpolation norm relation-equality spec radius}
    r_\fB(a)=r_\fC(a)
  \end{equation}
  for every $a\in \fS_h$.
  \end{thm}

\begin{proof}
Fix $a\in \fS_h$ and let $\fS(a)$ be the $*$-algebra generated by $a$ in $\fS$. We will let $\fA(a)$, $\fB(a)$ and $\fC(a)$ be the (commutative) Banach $*$-algebras generated by $a$ in $\fA$, $\fB$, and $\fC$, respectively. Then $(\fA(a), \fB(a),\fC(a))$ is a spectral interpolation of triple $*$-semisimple commutative Banach $*$-algebras relative to $\fS(a)$. Note $\fS(a)$ is quasi-Hermitian in $\fA(a)$ since
$$\partial \Sp_{\fA(a)}(a')\subseteq \Sp_{\fA}(a') \cup \{0\} \subseteq \R$$
for every $a'\in \fS(a)_h$ (see \cite[Proposition 5.12]{BD}) implies $\Sp_{\fA(a)}(a')\subseteq\R$ for all $a'\in \fS(a)_h$.
Therefore, $\fB(a)$ is a spectral subalgebra of $\fC(a)$ by Proposition \ref{P:interpolation triple-commutative alg} and, thus,
\begin{eqnarray*}
  r_\fB(a) = r_{\fB(a)}(a) = r_{\fC(a)}(a) = r_\fC(a).
\end{eqnarray*}
Hence, $\fS_h$ has invariant spectral radius in ($\fB$, $\fC$). The equality $C^*(\fB)=C^*(\fC)$ follows from Proposition \ref{P:dense Spec subalg-same C envelope}.
\end{proof}

\section{Quasi-Hermitian algebras associated to locally compact groups}\label{S:quasi-herm groups}

 We apply methods developed in the preceding section 
to Banach $*$-algebras associated to locally compact groups in this section. This culminates in a proof that  every quasi-Hermitian locally compact group is amenable.


Clearly every Hermitian group is also quasi-Hermitian. The converse is not true. A locally compact group $G$ is \emph{quasi-symmetric} if $\Sp_{L^1(G)}(f^**f)\subseteq [0,\infty)$ for every $f\in C_c(G)$ (see \cite[Definitions 5 and 6]{Pal2}). Palma shows every locally compact group with subexponential growth is quasi-symmetric in \cite[Proposition 3]{Palma2} by applying a result of Hulanicki (see \cite{Hul2}) and the Barnes-Hulanicki Theorem. 
Every quasi-symmetric locally compact group is quasi-Hermitian by the proof of Proposition \ref{T:Hermitian vs quasi Hermitian-comm alg} $(iii)\Longrightarrow(i)$. Therefore, taking into account the examples of non-Hermitian locally finite groups  described in Example \ref{E:Hermitian- weighted locally finite alg}, we state the following.

\begin{prop}\label{P:quasi Hermitian groups-exponential growth}
Every locally compact group with subexponential growth is quasi-Hermitian. In particular, the class of quasi-Hermitian locally compact groups is strictly larger than the class of Hermitian locally compact groups.
\end{prop}

Let $G$ be a locally compact group and $1\leq p \leq \infty$. The left-regular representation of $L^1(G)$ on $L^p(G)$ is given by
\begin{equation}\label{Eq:left reg rep-Lp}
  \lambda_p:L^1(G)\to B(L^p(G)) \ \ , \ \ \lambda_p(f)g=f*g.
\end{equation}
for all $f\in L^1(G)$ and $g\in L^p(G)$.

For $p=1$, this is nothing but an isometric embedding of $L^1(G)$ as convolution operators over itself. For $1\leq p\leq\infty $, the norm-closure of $\lambda_p(L^1(G))$ inside $B(L^p(G))$ is denoted by $\pf_p(G)$ and called the {\it algebra of $p$-pseudofunctions} on $G$. Clearly we have $\pf_2(G)=C^*_r(G)$.

Suppose $1\leq p\leq q\leq\infty$ are conjugate, i.e. $1/p+1/q=1$, and consider the (conjugate) duality relation $L^p(G)^*\cong L^q(G)$ given by
$$\la f , g \ra:=\int_G f(s)\overline{g(s)}ds \ \ \ (f\in L^p(G),g\in L^q(G)).$$
Then, for every $f\in L^1(G)$, $g\in L^p(G)$ and $h\in L^q(G)$,
\begin{eqnarray*}
  \la \lambda_p(f^*)g, h \ra &=& \la f^**g , h \ra \\
   &=& \int_G \int_G f^*(s)g(s^{-1}t)\overline{h(t)}dsdt  \\
   &=& \int_G \int_G f^*(s^{-1})\Delta(s^{-1})g(st)\overline{h(t)}dsdt \\
   &=& \int_G \int_G \overline{f(s)}g(st)\overline{h(t)}dsdt \\
   &=& \int_G \int_G \overline{f(s)}g(t)\overline{h(s^{-1}t)}dtds \\
   &=& \int_G g(t) \overline{\int_G f(s)h(s^{-1}t)ds} dt \\
   &=& \int_G g(t)\overline{f*h(t)} dt \\
   &=& \la g , \lambda_q(f)h \ra.
\end{eqnarray*}
In short,
\begin{equation}\label{Eq:anti isometry-Lp Lq-Pseuduefunctions}
  \la \lambda_p(f^*)g, h \ra = \la g , \lambda_q(f)h \ra.
\end{equation}
for every $f\in L^1(G)$, $g\in L^p(G)$ and $h\in L^q(G)$.
The relation \eqref{Eq:anti isometry-Lp Lq-Pseuduefunctions} may be compared with \cite[Note 2]{Branes} and its proof.

\begin{prop}
 Let $G$ be a locally compact group, and suppose $1\leq p, q \leq\infty$ satisfies $1/p+1/q=1$. The group algebra $L^1(G)$ is a normed $*$-algebra with respect to
   \begin{equation}\label{Eq:involutive pseudofunctions-norm}
  \|f\|_{\pf_p^*(G)}:=\max \{\|\lambda_p(f)\|_{B(L^p(G))},\|\lambda_q(f)\|_{B(L^q(G))} \} \ \ \ (f\in L^1(G)),
\end{equation}
and the standard convolution and involution on $L^1(G)$.
\end{prop}

\begin{proof}
Observe that $\|\cdot\|_{\pf_p^*(G)}$ is well-defined on $L^1(G)$ since $\|f\|_{\pf_p^*(G)}\leq \|f\|_1$ for every $f\in L^1(G)$.
It is straightforward to verify that $(L^1(G),\|\cdot\|_{\pf_p^*(G)})$ is a normed algebra. We now must show the involution on $L^1(G)$ is an isometry with respect to $\|\cdot\|_{\pf_p^*(G)}$. Let $f\in L^1(G)$. It follows from \eqref{Eq:anti isometry-Lp Lq-Pseuduefunctions} that
$$\sup \{ | \la \lambda_p(f^*)g, h \ra| : g\in L^p(G), h\in L^q(G), \|g\|_p\leq 1,\|h\|_q\leq 1 \} \leq \|\lambda_q(f)\|_{B(L^q(G))}.$$
Therefore
$$\|\lambda_p(f^*)\|_{B(L^p(G))} \leq \|\lambda_q(f)\|_{B(L^q(G))}.$$
By switching $p$ and $q$, we will see that
$$\|\lambda_q(f^*)\|_{B(L^q(G))} \leq \|\lambda_p(f)\|_{B(L^p(G))}.$$
Combining the preceding inequalities with \eqref{Eq:involutive pseudofunctions-norm}, we get
\begin{equation}\label{Eq:involutive isometry-pseudofunctions norm}
  \|f^*\|_{\pf_p^*(G)}\leq \|f\|_{\pf_p^*(G)}.
\end{equation}
We obtain the reverse inequality by replacing $f$ with $f^*$ in \eqref{Eq:involutive isometry-pseudofunctions norm}.
\end{proof}


\begin{defn}\label{D:involutive pseudofunctions}
  Let $G$ be a locally compact group, and let $1\leq p \leq\infty$. The Banach $*$-algebra $\pf_p^*(G)$ is defined to be the completion of $(L^1(G),\|\cdot\|_{\pf_p^*(G)},*)$.

\end{defn}
\noindent It is immediate that if $q$ is the conjugate of $p$, then $\pf_p^*(G)=\pf_q^*(G)$ isometrically as Banach $*$-algebras. These algebras have been considered by Kasparov-Yu (see \cite{KY}) and Liao-Yu (see \cite{LY}) in relation to Baum-Connes conjecture.

The next proposition demonstrates the importance of the preceding definition in our context (see also \cite[Proposition 2.4]{LY}). We first recall the following standard result from complex interpolation.

\begin{lem}
	Let $1\leq p_1<p_2< p_3 \leq \infty$. If $T$ is an operator that acts continuously on both $L^{p_1}(G)$ and $L^{p_3}(G)$, then $T$ also acts continuously on $L^{p_2}(G)$. Further,
	\begin{equation}\label{Eq:interpolation-Lp operators}
	\|T\|_{B(L^{p_2}(G))}\leq  \|T\|_{B(L^{p_1}(G))}^{1-\theta} \|T\|_{B(L^{p_3}(G))}^\theta
	\end{equation}
	for $0<\theta<1$ satisfying
	\begin{equation*}\label{Eq:interpolation-Lp operators-relations}
	\frac{1}{p_2}=\frac{1-\theta}{p_1}+\frac{\theta}{p_3}.
	\end{equation*}
\end{lem}

\begin{prop}\label{P:Interpolation triple-pseudofunctions}
	Let $G$ be a locally compact group and $1\leq p_1<p_2<p_3\leq 2$. Then $$(\pf^*_{p_1}(G),\pf^*_{p_2}(G),\pf^*_{p_3}(G))$$ is a spectral interpolation of triple $*$-semisimple Banach $*$-algebras relative to $\pf_{p_1}^*(G)$.
\end{prop}

\begin{proof}
	Let $q_1$, $q_2$, and $q_3$ be the conjugates of $p_1$, $p_2$, $p_3$, and choose $\theta\in (0,1)$ such that
	$$\frac{1}{p_2}=\frac{1-\theta}{p_1}+\frac{\theta}{p_3}.$$ Then, for $f\in L^1(G)$,
	\begin{eqnarray*}
		\|\lambda_{p_2}(f)\|_{B(L^{p_2}(G))} &\leq & \|\lambda_{p_1}(f)\|_{B(L^{p_1}(G))}^{1-\theta}\|\lambda_{p_1}(f)\|_{B(L^{p_3}(G))}^\theta\\
		&\leq &\|f\|_{\pf_{p_1}^*(G)}^{1-\theta}\|f\|_{\pf_{p_3}^*(G)}^{\theta}
	\end{eqnarray*}
	by Equation \eqref{Eq:interpolation-Lp operators}.
	Similarly,
	$$\|\lambda_{q_2}(f)\|_{B(L^{q_2}(G))}\leq \|f\|_{\pf_{p_1}^*(G)}^{1-\theta}\|f\|_{\pf_{p_3}^*(G)}^{\theta}$$
	since
	\begin{eqnarray*}
	\frac{1}{q_2}&=& 1-\frac{1}{p_2}\\
	&=& 1-\frac{1-\theta}{p_1}-\frac{\theta}{p_3}\\
	&=& 1-(1-\theta)\left(1-\frac{1}{q_1}\right)-\theta \left(1-\frac{1}{q_3}\right)\\
	&=& \frac{1-\theta}{q_1}+\frac{\theta}{q_3}.
	\end{eqnarray*}
	Hence,
	\begin{equation}\label{E:PF-interpolation}
		\|f\|_{\pf_{p_2}^*(G)}\leq \|f\|_{\pf_{p_1}^*(G)}^{1-\theta}\|f\|_{\pf_{p_3}^*(G)}^{\theta}
	\end{equation}
	for every $f\in L^1(G)$. As such, it suffices to show $\pf^*_{p_1}(G)\subseteq \pf^*_{p_2}(G)\subseteq \pf^*_{p_3}(G))$ is a nested triple of $*$-semisimple $*$-subalgebras.
	
	Let $1< p<2$ be arbitrary. Then, for $f\in L^1(G)$,
	\begin{eqnarray*}
		\|f\|_{\pf_p^*(G)} &=& \max \{\|\lambda_p(f)\|_{B(L^p(G))},\|\lambda_q(f)\|_{B(L^q(G))} \} \\ &\geq & \|\lambda_p(f)\|_{B(L^{p}(G))}^{1-\theta}
		\|\lambda_q(f)\|_{B(L^{q}(G))}^\theta \\
		&\geq&  \|\lambda_2(f)\|_{B(L^2(G))}\\
		&=& \|f\|_{C^*_r(G)}
	\end{eqnarray*}
	by Equation \eqref{Eq:interpolation-Lp operators}, where $q$ is the conjugate of $p$ and $0<\theta<1$ satisfies
	\begin{equation*}
	\frac{1}{2}=\frac{1-\theta}{p}+\frac{\theta}{q}.
	\end{equation*}
	So the identity map on $L^1(G)$ extends to a contraction $\pi_p : \pf_p^*(G)\to C^*_r(G)$. Furthermore, this map is injective since $C_c(G)$ is dense in $L^p(G)$, $L^q(G)$, and $L^2(G)$. Indeed, suppose that $T\in \ker \pi_p$. Choose a sequence $\{f_n\}\subset L^1(G)$ so that $\lim_{n\to \infty} f_n=T$ in $\pf_p^*(G)$. Since $T\in \ker \pi_p$, for every $g\in C_c(G)$, we have $\lim_{n\to \infty} f_n*g=0$ in $L^2(G)$. In particular, $f_n*g$ converges to 0 in measure for every $g\in C_c(G)$. Observe that the map
$$L^1(G)\to B(L^p(G) \oplus^\infty L^q(G)) \ \ , \ \ f\mapsto \lambda_p(f)\oplus \lambda_q(f)$$
extends to an isometric representation $\pf_p^*(G)\to B(L^p(G)\oplus^\infty L^q(G))$. As such, we view $\pf_p^*(G)$ as being contained in $B(L^p(G)\oplus^\infty L^q(G))$.
Considering $g_1$ and $g_2$ in $C_c(G)$, we then have
$$T(g_1,g_2)=(h_1,h_2),$$
where $h_1$ and $h_2$ are the limits of $f_n*g_1$ and $f_n*g_2$ in $L^p(G)$ and $L^q(G)$, respectively. Since both of these exist and converge to zero in measure, $h_1=0$ in $L^p(G)$ and $h_2=0$ in $L^q(G)$. Thus $T=0$ by norm density of $C_c(G)\oplus C_c(G)$ in $L^p(G)\oplus^\infty L^q(G)$.


	Now let $1\leq p<p'<2$. Substituting $(p,p',2)$ for $(p_1,p_2,p_3)$ in Equation \eqref{E:PF-interpolation} gives
	$$\|f\|_{\pf^*_{p'}(G)}\leq \|f\|_{\pf_{p}^*(G)}^{1-\theta}\|f\|_{C^*_r(G)}^\theta\leq \|f\|_{\pf_p^*(G)}^{1-\theta}\|f\|_{\pf_p^*(G)}^{\theta}=\|f\|_{\pf_p^*(G)}.$$
	Hence, the identity map on $L^1(G)$ extends to a contraction $\pf_p^*(G)\to \pf_{p'}^*(G)$ for $1\leq p< p'\leq 2$. This map is also injective by injectivity of $\pi_p$ and commutativity of the following diagram.
$$\xymatrix{
  & \pf_p^*(G) \ar[dl] \ar[dr]^{\pi_p}
  \ar@{}[d]|-{\circlearrowleft} \\
  \pf_{p'}^*(G) \ar[rr]_{\pi_{p'}}
  && C^*_r(G)
}
$$
Finally we note that since $L^1(G)$ and $C^*_r(G)$ are already known to be $*$-semisimple, we deduce $\pf_p^*(G)$ is $*$-semisimple for $1\leq p\leq 2$.
 So,  we conclude
	$$(\pf^*_{p_1}(G),\pf^*_{p_2}(G),\pf^*_{p_3}(G))$$ is a spectral interpolation of triple $*$-semisimple Banach $*$-algebras relative to $\pf_{p_1}(G)$ for all $1\leq p_1<p_2<p_3\leq 2$.
\end{proof}

\begin{cor}\label{C:Interpolation triple-pseudofunctions-L1 case}
	If $1<p<2$, then
	$$ (L^1(G),\pf_p^*(G),C^*_r(G))$$
	is a spectral interpolation of triple $*$-semisimple Banach $*$-algebras relative to $L^1(G)$.
\end{cor}

Recall that if $G$ is a locally compact group with left haar measure $\nu$, then
$$\nu_{S}:=\lim_{n\to \infty} \nu(S^n)^\frac{1}{n}$$
exists and belongs to the interval $[1,\infty)$ for every nonempty, open, pre-compact subset $S$ of $G$ (see \cite[Theorem 1.5]{Palma1}). A locally compact group $G$ has subexponential growth exactly when $\nu_S=1$ for every such subset $S$ of $G$.

We are now equipped to prove the main theorem of this paper.

\begin{thm}\label{T:quasi-Hermitian implies amenable}
The following are equivalent for a locally compact group $G$.\\
$(i)$ $G$ is quasi-symmetric;\\
$(ii)$ $G$ is quasi-Hermitian;\\
$(iii)$ $r_{L^1(G)}(f)=r_{C^*_r(G)}(f)$ for every $f\in C_c(G)_h$.\\
$(iv)$ $\Sp_{L^1(G)}(f)=\Sp_{C^*_r(G)}(f)$ for every $f\in C_c(G)$.
\end{thm}

\begin{proof}
	As previously mentioned, the implication $(i)$ $\Longrightarrow$ $(ii)$ is given by the proof of Proposition \ref{T:Hermitian vs quasi Hermitian-comm alg} $(iii)$ $\Longrightarrow$ $(i)$. Further, $(iii)$ $\Longrightarrow$ $(iv)$ is a special case of the Barnes-Hulanicki Theorem (see Theorem \ref{T:Barnes Theorem}), and $(iii)$ $\Longrightarrow$ $(iv)$ is clear. We now must prove $(ii)$ $\Longrightarrow$ $(iii)$.

  Suppose $G$ is quasi-Hermitian and $1<p<2$. Then $(L^1(G),\pf_p^*(G),C^*_r(G))$ is a spectral interpolation of triple $*$-semisimple Banach $*$-algebras relative to $L^1(G)$ by the previous Corollary.
  Hence, Theorem \ref{T:interpolation triple-general alg} implies $C_c(G)_h$ has invariant spectral radius in $(\pf_p^*(G),C^*_r(G))$, i.e.
  \begin{equation}\label{Eq:spec. radious-pseudo vs reduced alg}
  r_{\pf_p^*(G)}(f)=r_{C^*_r(G)}(f) \ \ \ (f^*=f\in C_c(G)_h).
\end{equation}
Now suppose $f\in C_c(G)_h$ and $S$ is a pre-compact, open, symmetric subset of $G$ with $\supp f\subseteq S$. We will write $f^n$ to denote $\underbrace{f*\cdots*f}_{n\text{ times}}$ for $f\in L^1(G)$. For every $n\in \N$, we have
\begin{eqnarray*}
  \|f^{n}\|_{L^1(G)} &=& \|f^{n} 1_{S^n}\|_{L^1(G)} \\
  &\leq & \|f^{n}\|_{L^p(G)} \|1_{S^n}\|_{L^q(G)}  \\
   &\leq &  \|f^{n-1}\|_{B(L^p(G))}\|f\|_{L^p(G)} \nu(S^n)^{1/q}  \\
   &=&  \|f^{n-1}\|_{\pf_p^*(G)}\|f\|_{L^p(G)} \nu(S^n)^{1/q},
\end{eqnarray*}
where $\nu$ denotes the Haar measure of $G$.

Thus, condition \eqref{Eq:spec. radious-pseudo vs reduced alg} in conjunction with taking the $n^{th}$ root of both sides as $n$ tends towards infinity implies
\begin{equation*}
r_{L^1(G)}(f)\leq  r_{\pf_p^*(G)}(f)\nu_S^{1/q}=r_{C_r^*(G)}(f) \nu_S^{1/q}
\end{equation*}
for $f\in C_c(G)_h$ by the spectral radius formula.
Taking limits as $q\to \infty$ (i.e., $p\to 1^+$), the preceding equation becomes
\begin{equation*}
r_{L^1(G)}(f)\leq  r_{C_r^*(G)}(f) \lim_{q\to \infty} \nu_S^{1/q}=r_{C_r^*(G)}(f)\leq r_{L^1(G)}(f).
\end{equation*}
Thus
\begin{equation*}
r_{L^1(G)}(f)=r_{C_r^*(G)}(f)
\end{equation*}
for every $f\in C_c(G)_h$.
\end{proof}

\begin{cor}\label{C:MainCor}
	A quasi-Hermitian locally compact group is amenable.
\end{cor}

\begin{proof}
	If $G$ is quasi-Hermitian, then $C_c(G)_h$ has invariant spectral radius in $(L^1(G),C^*_r(G))$ so that, by Proposition \ref{P:dense Spec subalg-same C envelope}, $C^*_r(G)$ is the C$^*$-envelope of $L^1(G)$. Hence, $G$ is amenable.
\end{proof}


\begin{rem}\label{R:free sub semigroup-not qH}
	Let $G$ be a discrete group containing a free sub-semigroup on two generators.
	Jenkins proved $G$ is not Hermitian by exhibiting a function $f\in c_c(G)_h$ on the group with $\Sp_{L^1(G)}(f)\not\subseteq\R$ (see \cite{Jenk1}). In particular, $G$ is not quasi-Hermitian. However, Jenkins' proof was long and complicated. Palmer gave a much simpler proof of Jenkins' result that $G$ is not Hermitian by appealing to properties of the spectral radius for C*-algebras (see \cite[Theorem 12.5.18 (f)]{Pal2}). We now give a short verification that $G$ is not quasi-Hermitian by appealing to Theorem \ref{T:quasi-Hermitian implies amenable} and some of Palmer's observations. We reproduce most of Palmer's argument for the sake of self-containment.
	
	Suppose $s,t\in G$ generate a free semigroup,  $a_0,a_1,a_2\in\C$ satisfy $|a_0|=|a_1|=|a_2|=\frac{1}{3}$ and
	$$\sup\{|a_0+a_1z+a_2z^2| : z\in\T\}<1.$$
	Then $r_{\ell^1(G)}(a_0\delta_e+a_1\delta_s+a_2\delta_{s^2})<1$ by the spectral mapping theorem and maximum modulus principle. Since $a_0\delta_e+a_1\delta_s+a_2\delta_{s^2}$ is Hermitian, we deduce
	$$ \|a_0\delta_e+a_1\delta_s+a_2\delta_{s^2}\|_{C^*_r(G)}=r_{C^*_r(G)}(a_0\delta_e+a_1\delta_s+a_2\delta_{s^2})<1.$$
	So
	$$\|(a_0\delta_e+a_1\delta_s+a_2\delta_{s^2})*\delta_t\|_{C^*_r(G)}=\|a_0\delta_t+a_1\delta_{st}+a_2\delta_{s^2t}\|_{C^*_r(G)}<1$$
	since $\delta_t$ is a unitary in $C^*_r(G)$. When we expand the $n$-fold convolution product
	$$ (a_0\delta_t+a_1\delta_{st}+a_2\delta_{st^2})^n,$$
	none of the $3^n$ terms coincide as $s$ and $t$ generate a free semigroup. Hence,
	$$\|(a_0\delta_t+a_1\delta_{st}+a_2\delta_{st^2})^n\|_{\ell^1(G)}=1$$
	for every $n\in \N$ and, so, $r_{\ell^1(G)}(a_0\delta_t+a_1\delta_{st}+a_2\delta_{st^2})=1$. Thus,
	$$\Sp_{\ell^1(G)}(a_0\delta_t+a_1\delta_{st}+a_2\delta_{st^2})\neq \Sp_{C^*_r(G)}(a_0\delta_t+a_1\delta_{st}+a_2\delta_{st^2})$$
	and, hence, $G$ is not quasi-Hermitian by Theorem \ref{T:quasi-Hermitian implies amenable} $(iv)$.
\end{rem}

\begin{rem}
	A locally compact group $G$ is quasi-Hermitian if and only if every compactly generated open subgroup of $G$ is quasi-Hermitian. Indeed, if $H$ is a compactly generated subgroup of a quasi-Hermitian group $G$, then $H$ is quasi-Hermitian by Theorem \ref{T:quasi-Hermitian implies amenable} $(iii)$ since $L^1(H)$ embeds isometrically into $L^1(G)$ and $C^*_r(H)$ embeds isometrically into $C^*_r(G)$. Now suppose every compactly generated subgroup of $G$ is quasi-Hermitian and $f\in C_c(G)_h$ is nonzero. Let $H$ be the open subgroup generated by the support of $f$. Then
	$$ \Sp_{L^1(G)}(f)\subseteq \Sp_{L^1(H)}(f)\subseteq\R.$$
	So $G$ is quasi-Hermitian.
	
	Note the above equivalence is not true if we replace the condition ``quasi-Hermitian'' with ``Hermitian'' since there exist locally finite non-Hermitian groups (see Example \ref{E:Hermitian- weighted locally finite alg}).
\end{rem}

As a consequence of the above remarks, we arrive at the following characterization amongst elementary amenable discrete groups.

\begin{cor}\label{C:qH-elementary amenable}
	An elementary amenable discrete group $G$ is quasi-Hermitian if and only if it is of subexponential growth.
\end{cor}

\begin{proof}
	If $G$ is of subexponential growth, then $G$ is quasi-Hermitian by \cite[Proposition 3]{Palma2}. Otherwise, $G$ contains a finitely generated subgroup of exponential growth and, hence, a free sub-semigroup on two generators (see \cite{Chou}). So $G$ is not quasi-Hermitian by Remark \ref{R:free sub semigroup-not qH}.
\end{proof}

In the finitely generated case, the preceding corollary gives the characterization that a finitely generated elementary amenable discrete group is Hermitian if and only if it is almost nilpotent since every finitely generated elementary amenable discrete group is either almost nilpotent or is of exponential growth (see \cite{Chou}) and almost nilpotent groups are Hermitian (see \cite[Theorem 12.5.17]{Palma2} or \cite{Los}).
It remains an open problem whether every (quasi-)Hermitian discrete group is of subexponential growth (see \cite[Question 2]{Palma2}).

Let us consider one additional problem: When is $C_c(G)$ quasi-Hermitian in $\pf^*_p(G)$ for $1<p<2$? If $G$ is a locally compact group such that  $C_c(G)$ is quasi-Hermitian in $\pf_p^*(G)$ for every $1<p<2$, then Theorem \ref{T:interpolation triple-general alg} and Proposition \ref{P:Interpolation triple-pseudofunctions} imply $r_{\pf_p^*(G)}$ coincides with $r_{C^*_r(G)}$ on $C_c(G)_h$ for all $1<p<2$. The proofs of Theorem \ref{T:quasi-Hermitian implies amenable} and Corollary \ref{C:MainCor} then show $G$ is amenable.

\begin{prop}
If $G$ is a locally compact group such that $C_c(G)$ is quasi-Hermitian in $\pf_p^*(G)$ for every $1<p<2$, then $G$ is amenable.
\end{prop}

We conjecture $C_c(G)$ is not quasi-Hermitian in $\pf^*_p(G)$ for every non-amenable locally compact group $G$ and $1<p<2$. This paper concludes by verifying this for groups with the rapid decay property and Kunze-Stein groups. Let us first recall the definitions of such groups.

Let $G$ be a locally compact group. It is a classical result that $L^2(G)$ is an algebra with respect to convolution product if and only if $G$ is compact. Groups with the rapid decay property and Kunze-Stein groups are classes of locally compact groups for which $L^2(G)$ is closed under left convolution by certain classes of functions which are ``close to'' $L^2(G)$.

Fix a locally compact group $G$. For a length function $\ell$ on $G$ and $\alpha>0$, we let $\omega_{\ell,\alpha}$ denote the weight on $G$ given by
$$ \omega_{\ell,\alpha}(s):=(1+\ell(s))^\alpha$$
for $s\in G$. The locally compact group $G$ has the \emph{rapid decay property} if $L^2(G)$ is closed under convolution on the left by functions from $L^2(G,\omega_{\ell,\alpha})$ for some length function $\ell$ on $G$ and $\alpha>0$. If $G$ has the rapid decay property, then the closed graph theorem shows $L^2(G,\omega_{\ell,\alpha})$ embeds continuously inside $C^*_r(G)$. The rapid decay property was first established for free groups by Haagerup (see \cite{Haag}) and later formalized by Jolissaint (see \cite{Jol}). It is known that an amenable group has rapid decay with respect to a length function $\ell$ if and only if it has polynomial growth with respect to $\ell$ (see again \cite{Jol}). Groups with the rapid decay property are a very interesting class of locally compact groups for which many difficult problems become more tractable. See \cite{Chat} for a nice survey.

In 1960, Kunze and Stein observed that if $G=SL(2,\R)$, then $L^2(G)$ is closed under convolution on the left by functions from $L^p(G)$ for every $1\leq p<2$ (see \cite{KS}). Cowling famously shows that this property is shared by every connected semisimple Lie group with finite centre and defines a \emph{Kunze-Stein group} to be a locally compact if $L^2(G)$ is closed under convolution on the left by functions from $L^p(G)$ for every $1\leq p<2$ in \cite{Cowling}. On the opposite spectrum, there are also many examples of totally disconnected Kunze-Stein groups acting on trees (see \cite{Neebia}). The closed graph theorem shows $L^p(G)$ embeds continuously inside $C^*_r(G)$ for every $1\leq p<2$ and Kunze-Stein group $G$. We finally note that compact groups are the only examples of amenable Kunze-Stein groups (see \cite{Cowling}).

\begin{lem}
	Let $G$ be a locally compact group and $1< p<q<\infty$, where $p$ and $q$ are conjugate. If $C_c(G)$ is quasi-Hermitian in $\pf_p^*(G)$ and $S$ is a nonempty, open, symmetric, pre-compact subset of $G$, then
	$$ \nu_S^{-\frac{1}{q}}\leq r_{C^*_r(G)}\left(h_S\right)$$
	where $h_S:=\frac{1_S}{\nu(S)}$.
\end{lem}

\begin{proof}
	Suppose $C_c(G)$ is quasi-Hermitian in $\pf_p^*(G)$. Then, by Theorems \ref{T:Barnes Theorem} and \ref{T:interpolation triple-general alg}, $\pf_{p'}(G)$ is a spectral subalgebra of $C^*_r(G)$ for each $p<p'<2$. So the argument used to prove Theorem \ref{T:quasi-Hermitian implies amenable} shows that for every $f\in C_c(G)$ with support contained in $S$,
	$$ r_{L^1(G)}(f)\leq r_{C^*_r(G)}(f)\nu_S^{\frac{1}{q'}},$$
	where $q'$ is the conjugate of $p'$. Allowing $p'$ to tend towards $p$ so that $q'$ tends towards $q$, we deduce
	$$r_{L^1(G)}(f)\leq r_{C^*_r(G)}(f)\nu_S^{\frac{1}{q}}$$
	for every $f\in L^1(G)$. Replacing $f$ with $h_S*h_S\in C_c(G)$ and $S$ with $S^2$, we conclude
	$$ 1=r_{L^1(G)}(h_S*h_S)\leq r_{C^*_r(G)}(h_S*h_S)\nu_{S^2}^{\frac{1}{q}}=r_{C^*_r(G)}(h_S)^2\nu_S^{\frac{2}{q}}.$$
This completes the proof.
\end{proof}

\begin{rem}
In the case when $G$ is a discrete group generated by a finite symmetric set $S$, the function $h_S$ and its spectrum and spectral radius in $C^*_r(G)$ and $C^*(G)$ is heavily studied in various areas of mathematics such as geometric group theory or random walk on groups. For example, a celebrated and fundamental theorem of Kesten states that $1\in \Sp_{C_r^*(G)}(h_S)$ (or equivalently, $r_{C^*_r(G)}\left(h_S\right)=1$) if and only if $G$ is amenable (see \cite[Proposition 4.7]{dHGV2}). On the other hand, $G$ has Kazdhan property (T) if and only if $1$ is an isolated point of $\Sp_{C^*(G)}(h_S)$ (see \cite[Proposition III]{dHGV1}). We refer the intersted reader to \cite{dHGV1} and \cite{dHGV2} for more details and also generalizations.
\end{rem}

\begin{thm}
	Let $G$ be a locally compact group and $1<p<2$. If $G$ either
	\begin{itemize}
		\item[(a)] has the rapid decay property, or
		\item[(b)] is a Kunze-Stein group
	\end{itemize}
	and $C_c(G)$ is quasi-Hermitian in $\pf_p^*(G)$ for some $1<p<2$, then $G$ is amenable.
\end{thm}

\begin{proof}
		(a) Suppose that $G$ has the rapid decay property and $C_c(G)$ is quasi-Hermitian in $\pf_p^*(G)$ for some $1<p<2$. Choose a length function $\ell$ on $G$ and $\alpha>0$ so that $L^2(G,\omega_\alpha)\subseteq C^*_r(G)$. Let $q$ denote the conjugate of $p$, $K$ be the norm of the embedding map $L^2(G,\omega)\to C_r^*(G)$, and $S$ be any open, symmetric, pre-compact subset of $G$. The previous lemma implies
	$$ \nu_S^{-\frac{1}{q}}\leq K\|h_S\|_{2,\omega_\alpha}\leq KC_S\nu(S)^{-\frac{1}{2}},$$
	where
	$$ C_S:=\sup\{\omega_{\ell,\alpha}(s) : s\in S\}.$$
	Replacing $S$ with $S^n$, we deduce
	$$ \nu_{S^n}^{-\frac{1}{q}}\leq KC_{S^n}\nu(S^n)^{\frac{1}{2}}\leq n^\alpha KC_S \nu(S^n)^{-\frac{1}{2}}$$
	since $\ell$ is a length function. Taking the $n$th root of both sides and letting $n$ tend towards infinity, we get
	$$\nu_S^{-\frac{1}{q}}\leq \nu_S^{-\frac{1}{2}}.$$
	Since $q>2$, we deduce $\nu_S\leq 1$. So $G$ has subexponential growth and, hence, is amenable.
	
	(b) Suppose $G$ is a Kunze-Stein group and $C_c(G)$ is quasi-Hermitian in $\pf_p^*(G)$ for some $1<p<2$. Let $K_{p'}$ denote the norm of the canonical embedding $L^{p'}(G)\to C_r(G)$ for $1\leq p'<2$. If $S$ is an open, symmetric, pre-compact subset of $G$, then
	$$\nu_{S^n}^{-\frac{1}{q}}\leq K_{p'}\|h_{S^n}\|_{L^{p'}(G)}= K_{p'}\nu(S^n)^{-\frac{1}{q'}},$$
where $q'$ is the conjugate of $p'$. Taking the $n$th root of both sides and allowing $n$ to tend towards infinity, we deduce
	$$ \nu_{S}^{-\frac{1}{q}}\leq \nu_S^{-\frac{1}{q'}}.$$
	By choosing $p'$ to be in the interval $(p,2)$, we will have $q'<q$ so that we must have $\nu_{S}=1$. Thus $G$ has subexponential growth and must therefore be amenable.
\end{proof}


\subsection*{Acknowledgement}
The authors are grateful to Varvara Sheplenska for her careful reading of an early draft of this paper and conversations with her regarding this project. We would like to also thanks Yemon Choi and Michael Leinert for making comments on the earlier draft of this paper which led us to the clarifications of Propositions \ref{P:dense Spec subalg-same C envelope} and \ref{P:Interpolation triple-pseudofunctions}.

\end{document}